\documentclass[12pt]{amsart}
\usepackage{amscd}
\usepackage{graphicx}
\usepackage{lmodern,pst-node}
\def\frk{\frak}               

\def\Phi{{\frk n}}
\def\Phi{{\frk N}}
\def\opn#1#2{\def#1{\operatorname{#2}}} 
\opn\chara{char} \opn\length{\ell} \opn\pd{pd} \opn\rk{rk}
\opn\projdim{proj\,dim} \opn\injdim{inj\,dim} \opn\rank{rank}
\opn\depth{depth} \opn\grade{grade} \opn\height{height}
\opn\embdim{emb\,dim} \opn\codim{codim}

\opn\Tr{Tr} \opn\bigrank{big\,rank}
\opn\superheight{superheight}\opn\lcm{lcm}
\opn\trdeg{tr\,deg}
\opn\reg{reg} \opn\lreg{lreg} \opn\ini{in} \opn\lpd{lpd}
\opn\size{size}
\opn\div{div} \opn\Div{Div} \opn\cl{cl} \opn\Cl{Cl}
\opn\Spec{Spec} \opn\Supp{Supp} \opn\supp{supp} \opn\Sing{Sing}
\opn\Ass{Ass} \opn\Min{Min} \opn\Res{Res} \opn\mindeg{mindeg} \opn\GCD{GCD}
\opn\Ann{Ann} \opn\Rad{Rad} \opn\Soc{Soc}
\opn\Im{Im} \opn\Ker{Ker} \opn\Coker{Coker} \opn\Am{Am}
\opn\Hom{Hom} \opn\Tor{Tor} \opn\Ext{Ext} \opn\End{End}
\opn\Aut{Aut} \opn\id{id}

\opn\nat{nat}
\opn\pff{pf}
\opn\Pf{Pf} \opn\GL{GL} \opn\SL{SL} \opn\mod{mod} \opn\ord{ord}
\opn\Gin{Gin} \opn\Hilb{Hilb} \opn\indeg{indeg}
\opn\aff{aff} \opn\con{conv} \opn\relint{relint} \opn\st{st}
\opn\lk{lk} \opn\cn{cn} \opn\core{core} \opn\vol{vol}
\opn\link{link} \opn\star{star}
\opn\gr{gr}

\def\pot#1#2{#1[\kern-0.28ex[#2]\kern-0.28ex]}

\opn\dirlim{\underrightarrow{\lim}}
\opn\inivlim{\underleftarrow{\lim}}
\def\Implies{\ifmmode\Longrightarrow \else
        \unskip${}\Longrightarrow{}$\ignorespaces\fi}
\def\implies{\ifmmode\Rightarrow \else
        \unskip${}\Rightarrow{}$\ignorespaces\fi}
\def\iff{\ifmmode\Longleftrightarrow \else
        \unskip${}\Longleftrightarrow{}$\ignorespaces\fi}

\let\:=\colon
\newtheorem{Theorem}{Theorem}[section]
\newtheorem{Lemma}[Theorem]{Lemma}
\newtheorem{Corollary}[Theorem]{Corollary}
\newtheorem{Proposition}[Theorem]{Proposition}
\newtheorem{Remark}[Theorem]{Remark}

\newtheorem{Example}[Theorem]{Example}

\newtheorem{Definition}[Theorem]{Definition}

\let\epsilon\varepsilon
\let\phi=\varphi
\let\kappa=\varkappa
\def\qed{\ifhmode\textqed\fi
      \ifmmode\ifinner\quad\qedsymbol\else\dispqed\fi\fi}
\def\textqed{\unskip\nobreak\penalty50
       \hskip2em\hbox{}\nobreak\hfil\qedsymbol
       \parfillskip=0pt \finalhyphendemerits=0}
\def\dispqed{\rlap{\qquad\qedsymbol}}

\opn\dis{dis}
\def\pnt{{\raise0.5mm\hbox{\large\bf.}}}

\opn\Lex{Lex}

\begin{document}
\title{An Efficient Algebraic Criterion for Shellability}

\author{Imran Anwar$^{1}$, Zunaira Kosar$^{1}$, Shaheen Nazir$^{2}$}
 \thanks{ The authors would like to thank Higher Education Commission of Pakistan for supporting
 this research.\\ {\bf 1.} Abdus Salam School of Mathematical Sciences, G.C. University, Lahore, Pakistan.\\  {\bf 2.} Lahore University of Management Sciences, Pakistan.\\
   }
\email {imrananwar@sms.edu.pk, zunairakosar@gmail.com,
shaheen.nazir@lums.edu.pk}
 \maketitle

\begin{abstract}
In this paper, we give a new and efficient algebraic criterion for
the pure as well as non-pure  shellability of  simplicial complex
$\Delta$ over $[n]$. We also give an algebraic characterization of a
{\em leaf} in a simplicial complex (defined in \cite{F2}). Moreover,
we introduce the concept of Gallai-simplicial complex
$\Delta_{\Gamma}(G)$ of a finite simple graph $G$. As an
application, we show that the face ring of the Gallai simplicial
complex associated to tree is Cohen-Macaulay.
 \vskip 0.4 true cm
 \noindent
  {\it Key words } : {\em shellable simplicial complex, face ring of a simplicial complex, facet ideal, Cohen-Macaulay ring.}\\
 {\it 2000 Mathematics Subject Classification}: Primary 13P10, Secondary
13H10, 13F20, 13C14.\\

\end{abstract}
\section{Introduction}

The shellability of a simplicial complex is a well known
combinatorial property that carries strong algebraic interpretations
for instance see \cite{HH} and \cite{St}. Algebraic criterion for
the shellability of a simplicial complex has also been a reasonably
important subject, firstly introduced by A. Dress in  \cite{AD}.
Dress \cite{AD} showed that $\Delta$ is (non-pure) shellable in the
sense of Bj\"{o}rner and Wachs \cite{BW}, if and only if the face
ring $K[\Delta]$ is clean. Later on Herzog and Popescu  \cite{HP}
extended the concept for determining the shellability of
multicomplexes. The shellability criterion for multicomplexes was
further refined by Popescu \cite{Po}. Cleanness is well known to be
the algebraic counterpart of shellability for simplicial complexes.
Eagon and Reiner \cite{ER} gave a translation of the pure
shellability of a dual simplicial complex $\check{\Delta}$ on the
monomial generators of the Stanley-Reisner ideal
$I_{\mathcal{N}}(\Delta)$. Their algebraic translation gave birth to
an important class of ideals known as ideals with linear quotients
(Eagon-Reiner \cite{ER} called them as Dually shellable ideals). A
relatively new algebraic criterion for the shellability was given in
\cite{AR}, but it was surprisingly found defective, see \cite{AS}.\\
The aim of this paper is to give an efficient algebraic criterion of
shellability and draw attention towards finding more algebraic
properties of shellable complexes in the facet ideal theory. In this
paper, we give a new and the most efficient algebraic criterion for
the shellability of pure as well as non-pure simplicial complex
$\Delta$ in Theorem \ref{main} in terms of the monomial generators
of its facet ideal $I_{\mathcal{F}}(\Delta)$. We also give an
algebraic characterization of a {\em leaf} in a simplicial complex
in Theorem \ref{leaf}. In the last section, we use the concept of
Gallai graph $\Gamma(G)$ of a planar graph $G$ to introduce Gallai
simplicial complex $\Delta_{\Gamma}(G)$. The buildup of Gallai
simplicial complexes from a planar graph is an abstract idea,
somehow, similar to building an origami shape from a plane sheet of
paper by defining a crease pattern. We use a planar graph to build a
$2$-dimensional simplicial complex. We discuss the connectedness of
the Gallai simplicial complexes and give a characterization of its
facets. As an application, we show that the Gallai simplicial
complexes associated to trees are shellable.

\section{Basic Setup}
 A {\em simplicial complex} $\Delta$ on the vertex set $[n]$ is a subset of $2^{[n]}$
with the property that if $F\in \Delta$ and $G\subset F$, then $G\in
\Delta$. The members of $\Delta$ are called {\em faces} and the
maximal faces under inclusion are called {\em facets}. If
$\mathcal{F}(\Delta)=\{F_1, F_2, \ldots, F_s\}$ is the set of facets
of $\Delta$, we write $\Delta$ as $$\Delta=<F_1, F_2,\ldots,F_s>.$$
By a {\em subcomplex} of $\Delta$, we mean a simplicial complex
whose facet set is a subset of $\mathcal{F}(\Delta)$. We denote the
{\em dimension of a face} $F\in\Delta$ by $\dim(F)$ and it is
defined as $\dim(F)=\mid F\mid-1$. By the dimension of a simplicial
complex $\Delta$, we mean that $\dim(\Delta)= max\{\dim(F)\, \mid \,
F\, \hbox{is a {\em facet} in\, }\Delta \}$. We say that $\Delta$ is
a {\em pure simplicial complex} of dimension $d$, if  all the {\em
facets} of $\Delta$ are of dimension $d$.\\
\begin{Definition}\label{shell}{\em
A simplicial complex $\Delta$ over $[n]$ is shellable if its facets
can be ordered $F_1, F_2,\ldots, F_s$ such that, for all $2\leq
j\leq s$ the subcomplex
$$\hat{\Delta}_{<F_j>}=<F_{1},F_{2},\ldots,F_{j-1}>\cap<F_{j}>$$ is a pure of dimension $\dim(F_{j})-1$.
}
\end{Definition}
Shellability in the case of non-pure simplicial complexes was
firstly defined by Bj\"{o}rner and Wachs \cite{BW}.

Here, we recall the definition of connected simplicial complex from
\cite{F1}.
\begin{Definition}\label{connected}{\em A simplicial complex $\Delta$ is said to be connected
if for any two facets $F$ and $G$ of $\Delta$, there exists a
sequence of facets $F=F_0$, $F_1$,\ldots, $F_t=G$ such that $F_i\cap
F_{i+1} \neq \emptyset$, for any $i\in \{0,1,2,\ldots, t-1\}$. A
disconnected simplicial complex is that which is not connected or
equivalently if the vertex set $V$ of $\Delta$ can be written as
disjoint union of $V_1$ and $V_2$ such that no face of $\Delta$ has
vertices in both $V_1$ and $V_2$.  }
\end{Definition}

The following definitions serve as the bridge between algebra and
simplicial complexes.

\begin{Definition}\label{facering}{\em  Let $\Delta$ be a simplicial complex over $[n]$ and $S=k[x_1,\ldots,
x_n]$ be the polynomial ring over an infinite field $k$. Let
$I_{\mathcal{N}}(\Delta)$  be the ideal of $S$ minimally generated
by square-free monomials $x_{j_1}x_{j_2} \ldots x_{j_s}$, where
$\{{j_1}, {j_2}, \ldots , {j_s}\}\subset [n]$ is not a face of
$\Delta$. $I_{\mathcal{N}}(\Delta)$ is known as {\em non-face ideal}
or the {\em Stanley- Reisner ideal} of $\Delta$. The quotient ring
$S/I_{\mathcal{N}}(\Delta)$ is called the {\em face ring} of
$\Delta$ denoted by $k[\Delta]$. }\end{Definition}

\begin{Definition}\label{facet}{\em (Faridi \cite{F1}). Let $\Delta$ be a simplicial complex over $[n]$ and $S=k[x_1,\ldots,
x_n]$ be the polynomial ring over an infinite field $k$.  Let
$I_{\mathcal{F}}(\Delta)\subset S$ be the monomial ideal minimally
generated by square-free monomials $m_{{F_1}},\ldots, m_{{F_s}} $
such that $m_{{F_i}}=x_{i_1}x_{i_2} \ldots x_{i_r}$, where
$F_i=\{{i_1} , \ldots, {i_r}\}\subset[n]$ is a {\em facet} of
$\Delta$ for all $i\in\{1,\ldots,s\}$. $I_{\mathcal{F}}(\Delta)$ is
known as the {\em facet ideal} of $\Delta$. }
\end{Definition}

Here, we recall the definition of pure square-free monomial ideal
from \cite{im}.
\begin{Definition}{\em
Let  $I\subset S$ be a square-free monomial ideal with a
minimal\linebreak generating system $\{g_1,\ldots ,g_m\}$.  We say
that $I$ is a {\em pure square-free monomial ideal of degree $d$} if
and only if $\supp(I)=\{x_1,\ldots,x_n\}$\footnote{$\supp(I)=\{x_j\,
\mid\, x_j\hbox{\ divides }u, \hbox{\ with\ } u\in G(I)\}$} and
$\beta_{0j}(I)= 0$\footnote{graded betti number of the ideal $I$}
for all  $j\neq d$. }
\end{Definition}
We conclude this section with recalling following definitions from
\cite{AR}.
\begin{Definition}\label{min}{\em Let $I$ be a monomial
ideal in $S$. We define the $\indeg (I)$ as follows$$\indeg
(I)=min\{j\, : \beta_{0j}(I)\neq 0\}.$$

}
\end{Definition}

\begin{Definition}\label{GD}{\em
Let $I\subset S=k[x_1,\ldots,x_n]$ be a monomial ideal, we say that
$I$ has {\em quasi-linear quotients}, if there exists an ordered
minimal monomial system of generators $m_1, m_2, \ldots, m_r$ of $I$
such that $\indeg(\hat{I}_{m_i})=1$ for all $1<i\leq r$, where
$$\hat{I}_{m_i}=(m_1, m_2, \ldots, m_{i-1}):(m_i).$$ }
\end{Definition}

\section{Linear residuals and shellability }
In this section, we describe some new algebraic notion for
explaining algebraic criterion of  pure as well as non-pure
shellability of $\Delta$ in the sense of Bj\"{o}rner and Wachs
\cite{BW}.

\begin{Remark}{\em In \cite[Theorem 3.4]{AR}, it had been shown that  $\Delta$ will be a  pure shellable simplicial complex if
and only if $I_{\mathcal{F}}(\Delta)$ has  quasi-linear quotients.
But, in \cite{AS}, it was mentioned that the facet ideal
$I_{\mathcal{F}}(\Delta)=(x_1x_2x_3,x_2x_3x_4,x_3x_4x_5,x_4x_5x_1)$
of the pure simplicial complex $\Delta=\langle
\{1,2,3\},\{2,3,4\},\{3,4,5\},\{4,5,1\}\rangle$ has quasi-linear
quotients but $\Delta$ is not shellable. Therefore, if a simplicial
complex $\Delta$ is pure shellable then $I_\mathcal{F}(\Delta)$ has
quasi-linear quotients but not vice versa. }
\end{Remark}

The following definition is essential in describing our algebraic
criterion  for the shellability.

\begin{Definition}\label{Residual}{\em Let $I\subset S=k[x_1,x_2,\ldots,x_n]$ be a monomial
ideal. We say that $I$ has {\em\bf linear residuals} if there exists
an ordered minimal monomial system of generators $\{m_1,
m_2,\ldots,m_r\}$ of $I$ such that $\Res(I_i)$ is minimally
generated by linear monomials for all $1< i\leq r$, where
$\Res(I_{i})=\{u_1, u_2,\ldots, u_{i-1}\}$ such that \linebreak
$u_k=\frac{m_i}{\gcd(m_k,m_i)}$ for all $1\leq k\leq i-1$.}
\end{Definition}

Here is our main result of this section.
\begin{Theorem}\label{main}{\em
Let $\Delta$ be a simplicial complex of dimension $d$ over $[n]$.
Then $\Delta$ will be {\em shellable } if and only if
$I_{\mathcal{F}}(\Delta)$ has {\em linear residuals}.
 }
\end{Theorem}
\begin{proof}
Let $\Delta=<F_1, F_2, \ldots , F_s>$ be a simplicial complex over
$[n]$ of dimension $d$. Firstly, we show that
$$\dim(\hat{\Delta}_{<F_i>})=dim(F_i)-\indeg(\Res(I_{F_i})) \hbox{ \,  for all \, }2\leq i
\leq s,$$ where $m_{F_1}, m_{F_2}, \ldots, m_{F_s}$ is the minimal
monomial generating system of $I_{\mathcal{F}}(\Delta)$.\\
By \ref{Residual}, a monomial generating system of $\Res(I_{F_i})$
is given as:
$$\Res(I_{F_i})=\{u_1,u_2,\ldots,u_{i-1}\}$$ with
$u_k={\frac{m_{F_i}}{\gcd(m_{F_i},m_{F_k})}}$ for $1\leq k\leq i-1$.
Then   $x_j\mid u_k$ for some $ j\in [n]$ if and only if $\{j\}\in
F_i\setminus F_k$. Therefore, $$\deg(u_k)=\mid F_i\setminus F_k\mid
= \dim(F_i) - \dim(F_i\cap F_k) \, \hbox{for all\, } k < i.$$ It
implies that $\indeg(\Res(I_{F_i}))=\min\{\dim(F_i) - \dim( F_i\cap
F_k) \, \hbox{for all\, } k < i\}$. Hence, we have
$\indeg(\Res(I_{F_i}))=\dim(F_i)-\dim(<F_1, F_2,\ldots F_{i-1}>\cap
<F_i>)$.\\ Let us consider $\Delta$ be a shellable simplicial
complex of dimension $d$ over $[n]$. Then for all $2\leq j\leq s$
the subcomplex
$$\hat{\Delta}_{<F_j>}=<F_{1},F_{2},\ldots,F_{j-1}>\cap<F_{j}>$$ is pure of dimension $\dim(F_j)-1$.  From above, it implies that
$\indeg(\Res(I_{F_j}))=1$. Moreover, from the purity of
$\hat{\Delta}_{<F_j>}$, we have $(\Res(I_{F_j}))$ is minimally
generated by linear monomials for all $2\leq j\leq s$. Because, if
there exists a term $u_k$ with $\deg(u_k)>1$ in the minimal
generators of $\Res(I_{<F_j>})$, then it implies that
$|F_j\setminus\{F_k\cap F_j\}|>1$ and
$F_k\cap F_j$ is a facet, causing $\hat{\Delta}_{<F_j>}$ non-pure.\\
Conversely, let $I_{\mathcal{F}}(\Delta)$ has linear residuals, then
$(\Res(I_{F_j}))$ is minimally generated by linear monomials for all
$2\leq j\leq s$. It implies from above that
$\dim(\hat{\Delta}_{<F_j>})=\dim(F_j)-1$. If
$(\Res(I_{F_j}))=(x_{j_1},\ldots,x_{j_t})$ then the subcomplex
$\hat{\Delta}_{<F_j>}=\langle F_j\setminus\{j_1\},\ldots,
F_j\setminus\{j_t\}\rangle$ will be pure for all $1<j\leq s$. Hence
proved.
\end{proof}
The following corollary gives an equivalence of the two algebraic
criterions of  shellability or one can say that it is relating two
different algebraic properties.
\begin{Corollary}\label{cleannes}{\em The face ring of a simplicial
complex $\Delta$ over $[n]$ is clean if and only if
$I_{\mathcal{F}}(\Delta)$ has linear residuals. }
\end{Corollary}
\begin{proof} We know from \cite[Theorem $\S 4$]{AD}, the face ring $k[\Delta]$ is clean if and only if $\Delta$ is
shellable. Therefore, result follows from Theorem \ref{main}.
\end{proof}
Theorem \ref{main} can be useful in proving the Cohen-Macaulayness
of the face ring of a pure simplicial complex as follows.
\begin{Corollary}\label{cohenres}{\em If the facet ideal $I_{\mathcal{F}}(\Delta)$ of a pure simplicial
complex $\Delta$ over $[n]$ has linear residuals, then the face ring
$k[\Delta]$  is Cohen-Macaulay.}
\end{Corollary}
A leaf of a simplicial complex (introduced by Faridi \cite{F2}) is a
facet $F$ of $\Delta$ such that either $F$ is the only facet of
$\Delta$, or there exists a facet $G$ in $\Delta$, $G\neq F$, such
that $F\cap \acute{F}\subseteq F\cap G$ for every facet
$\acute{F}\in \Delta$, $\acute{F}\neq F$. A simplicial complex
$\Delta$ is a simplicial tree if $\Delta$ is connected and every
subcomplex $\acute{\Delta}$ contains a leaf. By a subcomplex, we
mean any simplicial complex of the form $\acute{\Delta}=\langle
F_{i1},\ldots, F_{ir}\rangle$, where $\{F_{i1},\ldots, F_{ir}\}$ is
a subset of the set of facets of $\Delta$.
\begin{Theorem}\label{leaf}{\em Let $I_\mathcal{F}(\Delta)=(m_{F_1},\ldots,m_{F_r})$ with $r >1$ be the facet ideal
of a simplicial complex $\Delta$. A facet $F_i$ of $\Delta$ will be
a leaf if and only if $(\Res(\hat{I}_{F_i}))$ is a principal ideal,
where,
$$\Res(\hat{I}_{F_i})=\{u_j=\frac{m_{F_i}}{\gcd(m_{F_i},m_{F_j})}\, |
\hbox{\, for all\, } i\neq j\}.$$
  }\end{Theorem}
  \begin{proof}
  Suppose $F_i$ is a leaf in $\Delta$, then there exists some facet $F_k$
  for $k\neq i$ in $\Delta$ such that $F_i\cap F_j\subseteq F_i\cap
  F_k$ for all $i\neq j$. It implies that $\gcd(m_{F_i},m_{F_j})$ divides
  $\gcd(m_{F_i},m_{F_k})$, yielding
  $\frac{m_{F_i}}{\gcd(m_{F_i},m_{F_j})}$ divisible by
  $\frac{m_{F_i}}{\gcd(m_{F_i},m_{F_k})}$ for all $j\neq i$, as
  required.\\
Conversely, suppose that for a facet $F_i\in \Delta$,
$(\Res(\hat{I}_{F_i}))$ is a principal ideal generated by a monomial
$u$. Then, $u=\frac{m_{F_i}}{\gcd(m_{F_i},m_{F_p})}$ for some $p\neq
i$, divides $\frac{m_{F_i}}{\gcd(m_{F_i},m_{F_q})}$ for all $i\neq q
\neq p$. Therefore, $\gcd(m_{F_i},m_{F_q})$ divides
$\gcd(m_{F_i},m_{F_p})$, hence we have $F_i\cap F_q\subseteq F_i\cap
F_p$ for all $q\neq i$, implies $F_i$ is a leaf.
  \end{proof}

\section{Gallai Simplicial Complexes}
From here on, $G$ denotes a finite simple graph on the vertex set
$V(G)=[n]$ and edge-set $E(G)$. The Gallai graph $\Gamma(G)$ of $G$
is a graph whose vertex set is the edge set $E(G)$; two distinct
edges of $G$ are adjacent in $\Gamma(G)$ if they are incident in $G$
but do not span a triangle in $G$. In \cite{L}, authors discussed
various combinatorial properties of Gallai and anti-Gallai graph for
various classes of graphs.

The following definition is a nice combinatorial
buildup.
\begin{Definition}\label{Gallai}{\em
The {\bf Gallai graph} $\Gamma(G)$ of a graph $G$ is the graph whose
vertex set is the edge set of $G$; two distinct edges of $G$ are
adjacent in $\Gamma(G)$ if they are incident in G but do not span a
triangle in $G$.}\end{Definition}
\begin{Example}{\em
Given below is a graph $G$ and its Gallai graph $\Gamma(G)$.

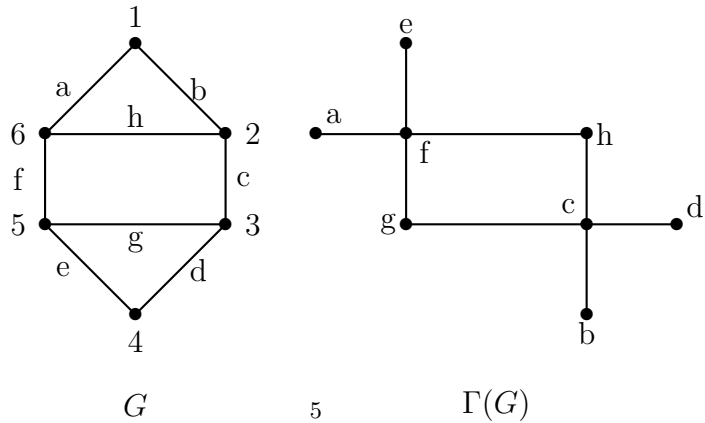
\begin{center}
\psset{unit=1.2cm}
\begin{pspicture}(0,0)(8,3)
 \psline(3,1)(2,2)
 \psline(3,0)(3,1)
 \psline(3,0)(2,-1)
 \psline(1,0)(2,-1)
 \psline(1,0)(1,1)
 \psline(2,2)(1,1)
 \psline(3,1)(1,1)
 \psline(3,0)(1,0)

 \psline(5,1)(7,1)
 \psline(5,0)(7,0)
 \psline(5,0)(5,1)
 \psline(7,0)(7,1)

 \psline(5,1)(4,1)
 \psline(7,-1)(7,0)
 \psline(5,2)(5,1)
 \psline(7,0)(8,0)

 \rput(5,1){$\bullet$}
 \rput(7,1){$\bullet$}
 \rput(5,0){$\bullet$}
 \rput(7,0){$\bullet$}
\rput(4,1){$\bullet$}
 \rput(7,-1){$\bullet$}
 \rput(5,2){$\bullet$}
 \rput(8,0){$\bullet$}

 \rput(1,0){$\bullet$}
 \rput(3,0){$\bullet$}
 \rput(3,1){$\bullet$}
 \rput(1,1){$\bullet$}
\rput(2,2){$\bullet$} \rput(2,-1){$\bullet$}

%

 \rput(2,2.3){1}

 \rput(0.7,0){5}

 \rput(2,-1.3){4}

 \rput(3.3,1){2}

 \rput(3.3,0){3}

 \rput(0.7,1){6}

 \rput(1.2,1.5){a}

\rput(2.7,1.5){b}

\rput(3.2,0.5){c}

\rput(2.7,-0.5){d}

\rput(1.2,-0.5){e}

\rput(0.7,0.5){f}

 \rput(2,1.2){h}

 \rput(2,-0.2){g}

 \rput(4.2,1.2){a}

\rput(7,-1.2){b}

\rput(6.8,0.2){c}

\rput(8.2,0.2){d}

\rput(5,2.2){e}

\rput(5.2,0.8){f}

 \rput(7.2,1){h}

 \rput(4.8,0){g}

\rput(2,-2){$G$}

 \rput(6,-2){$\Gamma(G)$}

 \end{pspicture}

 \end{center}
\,\,\,\,\,\,\,\,\,\,\,\,\,\,\,\,\,\,\,\,\,\,\,\,\,\,\,\,\,\, }
\end{Example}
The following definition is essence in the structural study of
Gallai graph $\Gamma(G)$.

\begin{Definition}\label{Omega}{\em Let $ G $ be a finite simple graph with vertex set $V(G) = [n]$ and edge set $E(G) = \{ e_{i,j} = \{i,j\} | i,j \in
V(G)\}$.\\
We define the {\bf set of Gallai-indices $\Omega(G)$} of the graph
$G$ as the collection of subsets of $V(G)$ such that if $e_{i,j}$
and $e_{j,k}$ are adjacent in $\Gamma(G)$, then
$F_{i,j,k}=\{i,j,k\}\in \Omega(G)$ or if $e_{i,j}$ is an isolated
vertex in $\Gamma(G)$ then $F_{i,j}=\{i,j\}\in \Omega(G)$.
}\end{Definition}

\begin{Definition}\label{Gsimp}{\em A {\bf Gallai simplicial complex $ \Delta_\Gamma(G)$ of $G$} is a simplicial complex defined over $V(G)$
 such that $$ \Delta_\Gamma(G)=<F\, |\, F \in
 \Omega(G)>, $$ where $\Omega(G)$ is the set of Gallai-indices of
 $G$.
}\end{Definition}
\begin{Example}{\em Let G be a given graph as below then its Gallai simplicial complex is as follow:

$\Delta_\Gamma(G) =
<\{1,2\},\{1,3,4\},\{1,3,5\},\{2,3,5\},\{3,5,6\},\{4,5,6\},\{2,3,4\}>$
 }
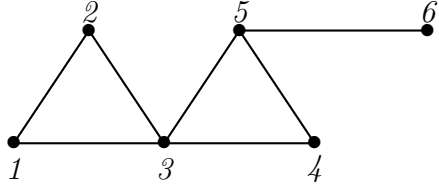
 \begin{center}

\psset{unit=0.5cm}

\begin{pspicture}(4,-1.5)(8,5)

 \psline(1,0)(5,0)

 \psline(5,0)(9,0)

 \psline(1,0)(3,3)

 \psline(5,0)(3,3)

 \psline(5,0)(7,3)

 \psline(9,0)(7,3)

 \psline(7,3)(12,3)

 \rput(1,0){$\bullet$}

 \rput(5,0){$\bullet$}

 \rput(9,0){$\bullet$}

 \rput(3,3){$\bullet$}

 \rput(7,3){$\bullet$}

 \rput(12,3){$\bullet$}

  \rput(1,-0.7){1}

 \rput(5,-0.7){3}

 \rput(9,-0.7){4}

 \rput(3,3.5){2}

 \rput(7,3.5){5}

 \rput(12,3.5){6}

 \end{pspicture}

 \end{center}

 \end{Example}
\begin{Proposition}\label{dim1}{\em
Let $G$ be a finite simple connected graph. Then the Gallai
simplicial complex $\Delta_\Gamma(G)$ of $G$ is one dimensional if
and only if $\Delta_\Gamma(G)=G$.}
\end{Proposition}
\begin{proof}
The dimension of Gallai simplicial complex $\Delta_\Gamma(G)$ of $G$
is one if and only if $\Delta_\Gamma(G)=\langle F\, |\, F \in
 \Omega(G) \hbox{\ and\ } |F|=2\rangle =\langle E(G)\rangle$ follows
from \ref{Omega}.
\end{proof}

\begin{Lemma}\label{Gind}{\em Let $G$ be a simple connected graph with vertex set $V(G)$. Let $\Omega(G)$ be the set of Gallai-indices of the graph
$G$, then for every $F=\{v_{1},v_{2},v_{3}\}\in \Omega(G)$ there
exists $H\in \Omega(G)$ such that $|F \cap H| = |H|-1$. }\end{Lemma}

\begin{proof}
We know from \ref{Omega}, that $F=E_i\cup E_j$ for some $E_i,E_j\in
E(G)$ and say $E_i\cap E_j=\{v_{1}\}$ and $\{v_{2},v_{3}\}\not\in
E(G)$. If $\deg(v_{2})=1=\deg(v_{3})$, then for any edge
$E_k=\{v_{1},v_{k}\}$, we have $H=\{v_{1},v_{2},v_{k}\}\in
\Omega(G)$ proving the result.\\  If $\deg(v_{2})\geq 2$, then for
any edge $E_m=\{v_{2},v_{m}\}$, we have
$H_1=\{v_{1},v_{2},v_{m}\}\in \Omega(G)$ provided
$\{v_{1},v_{m}\}\not\in E(G)$ proving the result, or
$H_2=\{v_{3},v_{1},v_{m}\}\in \Omega(G)$ provided
$\{v_{1},v_{m}\}\in E(G)$ but $\{v_{3},v_{m}\}\not\in E(G)$ proving
the result, or $H_3=\{v_{3},v_{m},v_{2}\}\in \Omega(G)$ and
$H_4=\{v_{1},v_{m}\}$ provided $\{v_{1},v_{m}\}\in E(G)$ and
$\{v_{3},v_{m}\}\in E(G)$ proving the result.
\end{proof}
Here, we give a small but important result about the connectedness
of Gallai simplicial complexes.
\begin{Lemma}\label{con}{\em Let $G$ be a simple connected graph, then its
Gallai simplicial complex $\Delta_\Gamma(G)$ will be
connected.}\end{Lemma}
\begin{proof}
Let $G$ be a simple connected graph then it is well known that for
any two vertices $v_j$ and $v_k$ there exists a sequence of edge-set
of $G$ as $\{E_0, E_1,\ldots E_r\}$ with $x_j\in E_0$ and $x_k\in
E_r$ such that $E_i\cap E_{i+1}\neq \emptyset$. If the Gallai
simplicial complex $\Delta_\Gamma(G)$ of a simple graph $G$ is of
dimension one then the result follows from
\ref{dim1}.\\
Now suppose $\dim(\Delta_\Gamma(G))=2$ for a simple connected graph
$G$. Let $F$ and $H$ be any two facets, with $|F\cap H|=\emptyset$.
Let us consider two vertices $v_r\in F$ and $v_s\in H$ of the
connected graph $G$. Therefore, there exists a sequence of connected
edges $E_j,E_{j+1},\ldots,E_k$. Then by \ref{Gsimp}, either
$E_{j+i-1}$ and $E_{j+i}$ yields a facet of $\Delta_\Gamma(G)$ as
$F_i=E_{j+i-1}\cup E_{j+i}\in \Omega(G)$ or giving two connected
facets $F_{i-1}$ and $F_{i}$ containing $E_{j+i-1}$ and $E_{j+i}$
respectively, proving the fact.
\end{proof}
Her we give a general shelling for Gallai simplicial complexes
associated to trees.
\begin{Theorem}
\em{The face ring of Gallai simplicial complex $\Delta_\Gamma(T)$
associated to a tree $T$ is Cohen-Macaulay.}\end{Theorem}
\begin{proof}
From $3.6$, it is sufficient to show that the facet ideal
$I_\mathcal{F}(\Delta_\Gamma(T))$ is pure and have linear residuals.
As tree $T$ does not contain any cycle, therefore
$I_\mathcal{F}(\Delta_\Gamma(T))$ is pure of the form.
$$\Omega(T)=\{\{{i_1},{i_2},{i_3}\}:
\hbox{ for any two adjacent edges \ } \{{i_1},{i_2}\} \hbox{ and\ }
\{{i_2},{i_3}\} \}.$$ It is well known that any two vertices in a
tree $T$ are connected by exactly one path. Without the loss of
generality, let us assume a path
$$P_1=\{\{v_1,v_2\},\{v_2,v_3\},\ldots,\{v_{m-1},v_m\}\}$$ such that
$\deg(v_1)=1=\deg(v_m)$. It gives rise to a subcomplex $$\langle
F_{1,2,3},F_{2,3,4},\ldots,F_{m-2,m-1,m}\rangle$$ of
$\Delta_\Gamma(T)$, where $F_{i,i+1,i+2}=\{v_i,v_{i+1},v_{i+2}\}$.
It is easy to see that the facet ideal of the subcomplex
$m_{\mathcal{F}(P_1)}=\{m_{F_{1,2,3}},m_{F_{2,3,4}},\ldots,
m_{F_{m-2,m-1,m}}\}$ has linear residuals, therefore, the subcomplex
is pure shellable followed from the Theorem \ref{main}. If
$\deg(v_j)=2$, for
all $2\leq j\leq m-1$, then we are done.\\
Otherwise, for any $v_j$ with $\deg(v_j)>2$, we have a path $P_2$
starting from $v_j$  ending at some vertex with degree $1$.
Therefore, we have
$$m_{\mathcal{F}(P_2)}=\{m_{F_{j-1,j,k}},m_{F_{j,j+1,k}},m_{F_{j,k,
k+1}},\ldots,\linebreak m_{F_{k_1-2,k_1-1, k_1}} \},$$  such that
$(m_{\mathcal{F}(P_1)}, m_{\mathcal{F}(P_2)} )$ has linear residuals
due to the fact that for any \linebreak $m\in m_{\mathcal{F}(P_1)}$,
we have $\gcd(m_{F_{j-1,j,k}},m)$ divides
$\gcd(m_{F_{j-1,j,k}},m_{F_{j-1,j,j+1}})$, therefore
$\Res(I_{m_{F_{j-1,j,k}}})$ is a principal ideal generated by
$\frac{m_{F_{j-1,j,k}}}{\gcd(m_{F_{j-1,j,k}},m_{F_{j-1,j,j+1}})}=x_k$.
Similarly, we apply the same order to all possible paths starting
from a vertex in $P_1$ and ending at some vertex of degree $1$.
Hence the facet ideal will have the linear residuals under the
following ordering of generators
$$I_\mathcal{F}(\Delta_\Gamma(T))=(m_{\mathcal{F}(P_1)},
m_{\mathcal{F}(P_2)},\ldots,m_{\mathcal{F}(P_r)} ).$$ This ordering
of generating set will yield linear residual regardless to the
ordering and labeling of $P_i$'s for all $i\geq 2$. Hence, the
result follows from \ref{main}.

\end{proof}

\end{document}